\documentclass[12pt]{amsart}
\usepackage{amssymb, amsmath}
\usepackage{url}
\usepackage{enumerate}

\newtheorem{theorem}{Theorem}[section]
\newtheorem{corollary}[theorem]{Corollary}
\newtheorem{prop}[theorem]{Proposition}
\newtheorem{lemma}[theorem]{Lemma}

\theoremstyle{remark}

\theoremstyle{definition}

\numberwithin{equation}{section}
\numberwithin{theorem}{section}

\newcommand{\alexc}{{\mathcal A}_c}

\newcommand{\N}{{\mathbb N}}
\newcommand{\R}{{\mathbb R}}

\newcommand{\fn}{\!:\!}

\providecommand{\abs}[1]{\lvert#1\rvert}
\providecommand{\norm}[1]{\lVert#1\rVert}

\newcommand{\intab}{\int_a^b}
\newcommand{\intone}{\int_{-1}^1}
\newcommand{\polym}{{\mathcal P}_m}
\newcommand{\polyn}{{\mathcal P}_n}
\newcommand{\poly}[1]{{\mathcal P}_{#1}}
\newcommand{\fdn}{f^{(n)}}
\newcommand{\phin}{\phi_n}
\newcommand{\psim}{\psi_m}
\newcommand{\phint}{{\tilde \phi_n}}
\newcommand{\Knpt}{{\tilde K_{n,p}}}

\begin{document}
\date{Preprint April 14, 2016.  To appear in {\it Journal of Classical
Analysis}.}
\title[Corrected trapezoidal rules]
{Higher order corrected trapezoidal rules in Lebesgue and Alexiewicz spaces}
\author{Erik Talvila}
\address{Department of Mathematics \& Statistics\\
University of the Fraser Valley\\
Abbotsford, BC Canada V2S 7M8}
\email{Erik.Talvila@ufv.ca}

\begin{abstract}
If $f\!:\![a,b]\to\R$ such that $f^{(n)}$ is integrable then integration
by parts gives the formula
\begin{align*}
&\intab f(x)\,dx  =  \\
&\frac{(-1)^n}{n!}\sum_{k=0}^{n-1}(-1)^{n-k-1}\left[
\phi_n^{(n-k-1)}(a)f^{(k)}(a)- \phi_n^{(n-k-1)}(b)f^{(k)}(b)\right]
+E_n(f),
\end{align*}
where $\phi_n$ is a monic polynomial of degree $n$ and the error is
given by
$E_n(f)=\frac{(-1)^n}{n!}\int_a^b f^{(n)}(x)\phi_n(x)\,dx$.   This then
gives a quadrature formula for $\int_a^bf(x)\,dx$.  The polynomial
$\phi_n$ is chosen to optimize the error estimate under the assumption
that $f^{(n)}\in L^p([a,b])$ for some $1\leq p\leq\infty$ or if
$f^{(n)}$ is integrable in the distributional or Henstock--Kurzweil sense.
Sharp error estimates are obtained.
It is shown that this formula is exact for all such $\phi_n$ if $f$ is a polynomial of
degree at most $n-1$.  If $\phi_n$ is a Legendre polynomial then the
formula is exact for $f$ a polynomial of degree at most $2n-1$.
\end{abstract}

\maketitle

\section{Introduction}\label{sectionintroduction}
This paper is based on the following observation.  Suppose we wish
to approximate the integral $\intab f(x)\,dx$.  If the $n$th derivative
of function $f$ is integrable then repeated integration by parts yields
the formula,
\begin{eqnarray}
\intab f(x)\,dx & = & \frac{(-1)^n}{n!}\sum_{k=0}^{n-1}(-1)^{n-k-1}\left[
\phi_n^{(n-k-1)}(a)f^{(k)}(a)- \phi_n^{(n-k-1)}(b)f^{(k)}(b)\right]\notag\\
 & & \qquad+E_n(f),
\end{eqnarray}
where $\phi_n$ is a monic polynomial of degree $n$ and 
$E_n(f)=\frac{(-1)^n}{n!}\intab f^{(n)}(x)\phi_n(x)\,dx$.  This then
gives a quadrature formula for the integral of $f$ with error term
$E_n(f)$.  

If $f^{(n)}\in L^p([a,b])$ then the H\"older inequality
gives the error estimate
$\abs{E_n(f)}\leq\norm{f^{(n)}}_p\norm{\phi_n}_q/n!$, where $p$ and $q$ 
are conjugate exponents.  (If $p,q\in(1,\infty)$ then $p^{-1}+q^{-1}=1$.
If $p=1$ then $q=\infty$.  If $p=\infty$ then $q=1$.  Hence, we define
$1/\infty=0$.)  A natural question
is then how to choose $\phi_n$ to minimize this error.  

For $n=1$ this
problem is completely solvable.  It is easy to see that the unique
minimizing polynomial is $\phi_1(x)=x-c$, where $c$ is the midpoint of
$[a,b]$. See Corollary~\ref{corollaryn1}.  The case $n=2$ was considered in
\cite{talvilawiersmajmi}.  It was shown there that the polynomial
that minimizes $\norm{\phi_2}_q$ is unique.  An explicit minimizing polynomial 
was found
for $p=1,2,4/3, \infty$.  This gave sharp estimates on the error for
these values of $p$, which improved on a number of error estimates 
that appear in the literature.  For other values of $p\in[1,\infty]$ good
estimates were found for the minimizing value of $\norm{\phi_2}_q$.  
See also \cite{krylov}.

In
the present paper we examine the case $n\geq 3$.  The minimizing
polynomial is shown to be unique and to have $n$ simple zeros in $[a,b]$.  
When $p=1$ it is the Chebychev polynomial of the first kind, $T_n$.
When $p=2$ it is the Legendre polynomial $P_n$.
When $p=\infty$ it is the Chebychev polynomial of the second kind, $U_n$.
Each of these is multiplied by a normalising factor so as to have leading 
coefficient one.  These polynomials are orthogonal on $[-1,1]$.  Each of
our polynomials is composed with a linear function that maps
the interval $[a,b]$ onto the interval $[-1,1]$.

We also consider the case when $\intab\fdn(x)\,dx$ exists as a
Henstock--Kurzweil integral.  This allows conditional convergence in the
error term and
includes the case of convergence
as an improper Riemann integral or as a Cauchy--Lebesgue integral.  A
suitable norm is then the Alexiewicz norm, given as
$\norm{g}=\sup_{a\leq x\leq b}\abs{\int_a^xg(t)\,dt}$.
The polynomial that minimizes the error is again shown to be 
unique and to have $n$ simple zeroes in $[a,b]$.  It is given by
$\phi_n(x)=2^{1-n}(T_n(x)-1)$, suitably modified by a linear transformation
as above.  The same formulas hold when $f^{(n-1)}$ is merely assumed
to be continuous.  Then $\fdn$ exists as a distribution and the error
integral exists as a continuous primitive integral.  For a discussion of the Henstock--Kurzweil integral and
Alexiewicz norm, see \cite{lee} or \cite{swartz}.  The continuous
primitive integral is discussed in \cite{talviladenjoy}.

The final section of the paper discusses the degree of exactness.
If $f$ is a polynomial of  degree at most $n-1$ then $E_n(f)=0$ for
all $\phin\in\polyn$.  If $\phin$ is a normalized Legendre polynomial of
degree $n$,
composed with a linear transformation as above, then $E_n(f)=0$ for all
polynomials $f$ of degree at most $2n-1$.

Several other authors have considered modified trapezoidal rules under
conditions on $\fdn$.  Cerone and Dragomir \cite{ceronedragomirtrap} 
assume $\fdn\in L^p$ and obtain formulas like \eqref{basic} but with
larger error coefficients than in this theorem or in Corollaries~
\ref{corollaryone}, \ref{corollarytwo}, \ref{corollaryinfinity}.  Similarly
with Dedi\'{c},  Mati\'{c} and  Pe\v{c}ari\'{c} in \cite{dedicmaticpecaric}.
Liu \cite{liu} assumes the condition $f^{(n-1)}\in C([a,b])
\cap BV([a,b])$ and has a quadrature formula with degree of exactness equal
to $n-1$.  The problem is tackled using the Peano kernel by Dubeau 
\cite{dubeau} and Pe\v{c}ari\'{c} and Ujevi\'{c} \cite{pecaricujevic}.
Ding, Ye and Yang \cite{dingyeyang} estimate the remainder when $f''$ is Henstock--Kurzweil 
integrable.

\section{$\fdn\in L^p([a,b])$}
Let $\polym$ denote the monic polynomials of degree $m$.
For $1\leq p<\infty$ let $L^p([a,b])$ be the Lebesgue measurable
functions such that $\norm{g}_p=(\intab \abs{g(x)}^p\,dx)^{1/p}<\infty$.
Let $L^\infty([a,b])$ be the essentially bounded functions, with
norm $\norm{g}_\infty={\rm ess\, sup}_{x\in[a,b]}\abs{g(x)}$.  All
measure-theoretic statements are with respect to Lebesgue measure.
\begin{theorem}\label{theoremapprox}
Let $n\in\N$.  Let $p,q\in[1,\infty]$ be conjugate exponents.
Let $f\fn[a,b]\to\R$ such that $f^{(n-1)}$ is absolutely continuous and
$\fdn\in L^p([a,b])$.  Let $\phin\in\polyn$.  Write
\begin{eqnarray}
\intab f(x)\,dx & = & \frac{(-1)^n}{n!}\sum_{k=0}^{n-1}(-1)^{n-k-1}\left[
\phi_n^{(n-k-1)}(a)f^{(k)}(a)- \phi_n^{(n-k-1)}(b)f^{(k)}(b)\right]\notag\\
 & & \qquad+E_n(f),\label{basic}
\end{eqnarray}
where
$E_n(f)=\frac{(-1)^n}{n!}\intab f^{(n)}(x)\phi_n(x)\,dx$.
Then 
\begin{equation}
\abs{E_n(f)}\leq \frac{\norm{\fdn}_p\norm{\phi_n}_q}{n!}
\leq K_{n,p}\norm{\fdn}_p(b-a)^{n+1/q}\label{errorestimate}
\end{equation}
for a constant $K_{n,p}$ that depends on $n$, $p$ and $\phin$ but is independent of
$f$ and $b-a$.
There is a unique polynomial $\phint\in\polyn$ that minimizes $K_{n,p}$.
The estimate on $\abs{E_n(f)}$ is then sharp in the sense that the
coefficient of $\norm{\fdn}_p$ cannot be reduced.
The minimum value of $K_{n,p}$ is $\Knpt=2^{-n-1/q}\norm{\phint}_q/n!$
where the norm of $\phint$ is taken over $[-1,1]$.
\end{theorem}
\begin{proof}
Integration by parts establishes \eqref{basic}.  The H\"older inequality
gives \eqref{errorestimate}.
Define $\phint\in\polyn$ by
$\phin(x)=[(b-a)/2]^n\phint([2x-a-b]/[b-a])$.
And write $\norm{\phint}_q=(\int_{-1}^1\abs{\phint(x)}^q\,dx)^{1/q}$
for $1\leq q<\infty$ and $\norm{\phint}_\infty=\max_{\abs{x}\leq 1}
\abs{\phint(x)}$.
Then for $1\leq q<\infty$,
$$
\norm{\phin}_q =
\left(\frac{b-a}{2}\right)^{n}
\left(\intab
\left|\phint\left(\frac{2x-a-b}{b-a}\right)\right|^q\,dx\right)^{1/q}=
\left(\frac{b-a}{2}\right)^{n+1/q}\norm{\phint}_q.
$$
Similarly when $q=\infty$.
This produces the factor $(b-a)^{n+1/q}$ in \eqref{errorestimate}.

Existence of a unique minimizing polynomial for $\norm{\phi_n}_q$ is
proved in Lemma~\ref{lemma}.

To show the coefficient of $\norm{\fdn}_p$ in \eqref{errorestimate}
 cannot be made any smaller,
note that
for $1<p<\infty$ there is equality in the H\"older inequality when
$\fdn(x)=d\,{\rm sgn}[\phin(x)]\abs{\phin(x)}^{1/(p-1)}$ for some
$d\in\R$ and almost all $x\in[a,b]$.  See \cite[p.~46]{liebloss}.
Integrate to get
\begin{eqnarray*}
f(x) & = & d\int_a^x\cdots\int_a^{x_{i+1}}\cdots\int_a^{x_2}{\rm sgn}
[\phin(x_1)]\abs{\phin(x_1)}^{1/(p-1)}dx_1\cdots dx_i\cdots dx_n\\
 & = & \frac{d}{(n-1)!}\int_a^x(x-t)^{n-1}{\rm sgn}[\phin(t)]
\abs{\phin(t)}^{1/(p-1)}\,dt,
\end{eqnarray*}
modulo a polynomial of degree at most $n-1$.  When $p=\infty$ the
condition for equality in the H\"older inequality is that
$\fdn(x)=d\,{\rm sgn}[\phin(x)]$ for some
$d\in\R$ and almost all $x\in[a,b]$.  See \cite[p.~46]{liebloss}.  We
can integrate as before to get
$$
f(x)=\frac{d}{(n-1)!}\int_a^x(x-t)^{n-1}{\rm sgn}[\phin(t)]\,dt,
$$
modulo a polynomial of degree at most $n-1$.  

When $p=1$ the
condition for equality in the H\"older inequality,
$\abs{\intone \fdn\phin}=\norm{\fdn}_1\norm{\phin}_\infty$, is that
$\phin(x)=d\,{\rm sgn}[\fdn(x)]$  for some
$d\in\R$ and almost all $x\in[-1,1]$.  See \cite[p.~46]{liebloss}.
(Because of the scaling argument above, it suffices to work on
$[-1,1]$.)
In general, this condition cannot be satisfied.
Take $\alpha\in[-1,1]$ such that $\norm{\phin}_\infty=\abs{\phin(\alpha)}$.
Let $\delta$ be the Dirac distribution.  If $\alpha\in(-1,1)$ and
$\fdn(x)=\delta(x-\alpha)$ then
$\intone\fdn(x)\phin(x)\,dx=\phin(\alpha)$.  Now,
$\delta(x-\alpha)\not\in L^1([-1,1])$ so use a $\delta$-sequence.  Let
$\psim\fn[-1,1]\to[0,\infty)$ be continuous with support in 
$(\alpha-1/m,\alpha+1/m)$
such that $\int_{\alpha-1/m}^{\alpha+1/m}\psim(x)\,dx=1$.  Let
\begin{eqnarray*}
f_m(x) & = & \int_{-1}^x\cdots\int_{-1}^{x_{i+1}}\cdots\int_{-1}^{x_2}
\psim(x_1)\,dx_1\cdots dx_i\cdots dx_n\\
 & = & \frac{1}{(n-1)!}\int_{-1}^x(x-t)^{n-1}\psim(t)\,dt.
\end{eqnarray*}
Then $f_m\in L^1([-1,1])$.  Note that
$$
\norm{f_m^{(n)}}_1=\intone\left|f_m^{(n)}(x)\,dx\right|=\intone\abs{\psim(x)}
\,dx=1.
$$
And, since $\phin$ is continuous, we get
$$
\lim_{m\to\infty}\left|\intone f_m^{(n)}(x)\phin(x)\,dx\right|
=\abs{\phin(\alpha)}=\norm{\phin}_\infty,
$$
thus showing that the coefficient of $\norm{\fdn}_1$ in \eqref{errorestimate}
cannot be
reduced.  If $\abs{\alpha}=1$ then for each $\varepsilon>0$ there is
$\beta\in(-1,1)$ such that $\norm{\phin}_\infty<\varepsilon+\abs{\phin(\beta)}$.
\end{proof}
Now we look at some special cases that can be solved completely.

When $n=1$ we get the usual trapezoidal rule.  See \cite{cruzuribeJIPAM}.
\begin{corollary}\label{corollaryn1}
If $n=1$ the approximation becomes $\intab f(x)\,dx\doteq
\frac{b-a}{2}\left[f(a)+f(b)\right]$ with sharp error
estimate 
$$
\abs{E_1(f)}  \leq  \left\{\begin{array}{cl}
\frac{1}{2}\norm{f'}_1(b-a), & p=1\\
\frac{1}{2}\left(\frac{1}{q+1}\right)^{1/q}\norm{f'}_p
(b-a)^{1+1/q}, & 
1<p<\infty\\
\frac{1}{4}\norm{f'}_\infty(b-a)^2, & p=\infty.
\end{array}
\right.
$$
\end{corollary}
\begin{proof}
The minimizing polynomial is $\phi_1(x)=x-c$, where $c$ is the
midpoint of $[a,b]$.
\end{proof}

The case $n=2$ is discussed in detail in \cite{talvilawiersmajmi}, where
$\phi_2$ and
the exact values of ${\tilde K_{2,p}}$ are found for $p=1,2, 4/3, \infty$.

As mentioned in the Introduction, the unique minimizing polynomial
for $\norm{\phin}_p$ is known when $p=1,2, \infty$.  For these cases
we can get an explicit form of the approximation to the integral that
minimizes the error and compute the exact value of $\Knpt$
from Theorem~\ref{theoremapprox}.  
Since $[a,b]$ is a compact interval we have
$L^\infty([a,b])\subseteq L^s([a,b])\subseteq L^r([a,b])\subseteq L^1([a,b])$ if $1\leq r
\leq s\leq\infty$.  The estimate for $p=1$ 
then applies when
$f\in L^r([a,b])$ for $1\leq r\leq \infty$.
The estimate for $p=2$ applies when
$f\in L^r([a,b])$ for $2\leq r\leq \infty$.
\begin{corollary}\label{corollaryone}
If $\fdn\in L^r([a,b])$ for some $1\leq r\leq \infty$ then 
\begin{align}
&\intab f(x)\,dx  \doteq\label{expansion1}\\
\frac{1}{(n-1)!}\sum_{k=0}^{n-1}&\left[
\frac{(b-a)^{k+1}(2n-k-2)!(n-k-1)!}{2^{2k+1}(2n-2k-2)!(k+1)!}\right]
\left[
f^{(k)}(a)+(-1)^{k}f^{(k)}(b)\right]\notag
\end{align}
with sharp error estimate
\begin{equation}
\abs{E_n(f)}\leq \frac{\norm{\fdn}_1(b-a)^n}{2^{2n-1}n!}.\label{errorone}
\end{equation}
\end{corollary}
\begin{proof}
The unique polynomial minimizing $\norm{\cdot}_\infty$ over $\polyn$ on 
$[-1,1]$ 
is $\phint(x)=2^{1-n}T_n(x)$, where $T_n$ is the Chebyshev polynomial of
first type.  See \cite[p.~63]{cheney}, \cite[p.~39]{natansonI} or 
\cite[p.~31, 45]{rivlin}.  Since $T_n(\cos\theta)=\cos(n\theta)$ we have
$\norm{\phint}_\infty=2^{1-n}$.  The minimizing polynomial on $[a,b]$ is
then $\phi_n(x)=(b-a)^n2^{1-2n}T_n([2x-a-b]/[b-a])$.  From the proof of
the Theorem, the sharp error estimate is then
$$
\abs{E_n(f)}\leq\frac{\norm{\fdn}_1\norm{2^{1-n}T_n}_\infty}{n!}\left(
\frac{b-a}{2}\right)^n=\frac{\norm{\fdn}_1(b-a)^n}{2^{2n-1}n!}.
$$

To compute the expansion in \eqref{basic} we need the derivatives of
$\phi_n$.  We have, $\phi_n^{(m)}(x)=2^{1+m-2n}(b-a)^{n-m}T_n^{(m)}(
[2x-a-b]/[b-a])$.  Derivatives of $T_n$ can be computed in terms of
Gegenbauer polynomials $C_n^\lambda$.  See 
\cite[8.949.2, 8.937.4]{gradshteyn}.  Then
\begin{eqnarray*}
\phi_n^{(m)}(b) & = & 2^{1+m-2n}(b-a)^{n-m}n\,2^{m-1}(m-1)!C_{n-m}^{m}(1)\\
 & = & \frac{(b-a)^{n-m}n(n+m-1)!m!}{2^{2n-2m-1}(n-m)!(2m)!}.
\end{eqnarray*}
Since $T_n(-x)=(-1)^nT_n(x)$ get $\phi_n^{(m)}(a)=(-1)^{n+m}\phi_n^{(m)}(b)$.
This gives \eqref{expansion1}.
\end{proof}
\begin{corollary}\label{corollarytwo}
If $\fdn\in L^r([a,b])$ for some $2\leq r\leq\infty$ then 
\begin{align}
&\intab f(x)\,dx  \doteq\label{expansion2}\\
\frac{n!}{(2n)!}\sum_{k=0}^{n-1}&\left[
\frac{(b-a)^{k+1}(2n-k-1)!}{(n-k-1)!(k+1)!}\right]
\left[
f^{(k)}(a)+(-1)^{k}f^{(k)}(b)\right]\notag
\end{align}
with sharp error estimate
\begin{equation}
\abs{E_n(f)}\leq \frac{\norm{\fdn}_2n!(b-a)^{n+1/2}}{(2n+1)^{1/2}(2n!)}.
\label{errortwo}
\end{equation}
\end{corollary}
\begin{proof}
The unique polynomial minimizing $\norm{\cdot}_2$ over $\polyn$ on 
$[-1,1]$ 
is $\phint(x)=2^{n}(n!)^2P_n(x)/(2n)!$, where $P_n$ is the Legendre 
polynomial.
See \cite[p.~109]{cheney}, \cite[p.~48]{natansonII} or 
\cite[p.~62]{rivlin}.  As in the proof of Corollary~\ref{corollaryone},
the error estimate follows from the integral
$\int_{-1}^1P_n^2(x)\,dx=2/(2n+1)$ \cite[22.2.10]{abramowitz}.

Derivatives of $P_n$ can be computed from the hypergeometric
representation $P_n(x)={_2}F_1(-n,n+1;1;(1-x)/2)$ 
\cite[22.5.49, 15.2.2]{abramowitz}.  We get
$$
P_n^{(m)}(x)= \frac{(-n)_m(n+1)_m(-1)^m}{(1)_m2^m}{_2}F_1(m-n,m+n+1;m+1;(1-x)/2),
$$
where $(a)_m$ is the Pochhammer symbol.  Since ${_2}F_1(a,b;c;0)=1$ we
get 
$$
P_n^{(m)}(1)= \frac{(n+m)!}{(n-m)!m!2^m}.
$$
This gives
$$
\phi_n^{(m)}(b)=\frac{(b-a)^{n-m}(n!)^2(n+m)!}{(2n)!(n-m)!m!}
$$
with $\phi_n^{(m)}(a)=(-1)^{n+m}\phi_n^{(m)}(b)$.  And, \eqref{expansion2}
follows.
\end{proof}
\begin{corollary}\label{corollaryinfinity}
If $\fdn\in L^\infty([a,b])$ then 
\begin{align}
&\intab f(x)\,dx  \doteq\label{expansioninfinity}\\
\frac{1}{n!}\sum_{k=0}^{n-1}&\left[
\frac{(b-a)^{k+1}(2n-k)!(n-k-1)!}{2^{2k+2}(2n-2k-1)!(k+1)!}\right]
\left[
f^{(k)}(a)+(-1)^{k}f^{(k)}(b)\right]\notag
\end{align}
with sharp error estimate
\begin{equation}
\abs{E_n(f)}\leq \frac{\norm{\fdn}_\infty (b-a)^{n+1}}{2^{2n}n!}.
\label{errorinfinity}
\end{equation}
\end{corollary}
\begin{proof}
The unique polynomial minimizing $\norm{\cdot}_1$ over $\polyn$ on 
$[-1,1]$ 
is $\phint(x)=2^{-n}U_n(x)$, where $U_n$ is the Chebyshev  
polynomial of second type.
See \cite[p.~222]{cheney}, \cite[p.~26]{golomb} or 
\cite[p.~72, 83]{rivlin}.
From the formula $U_n(\cos\theta)=\sin([n+1]\theta)/\sin\theta$ we can
directly compute 
$$
\norm{U_n}_1=\frac{1}{n+1}\int_0^{(n+1)\pi}\abs{\sin\theta}\,d\theta =2.
$$
Then $\norm{\phint}_1=2^{1-n}$ and the
error estimate \eqref{errorinfinity} follows as in the previous corollaries.

Derivatives of $U_n$ can be computed in terms of Gegenbauer polynomials
\cite[8.949.5, 8.937.4]{gradshteyn}.  We get
$$
U_n^{(m)}(1)=2^mm!\binom{n+m+1}{n-m}.
$$
And,
$$
\phi_n^{(m)}(b)=\frac{(b-a)^{n-m}m!(n+m+1)!}{2^{2n-2m}(2m+1)!(n-m)!}
$$
with $\phi_n^{(m)}(a)=(-1)^{n+m}\phi_n^{(m)}(b)$.  Expansion 
\eqref{expansioninfinity} now
follows.
\end{proof}

The minimizing polynomial for $\norm{\cdot}_p$ is even or odd about
the midpoint of $[a,b]$ as $n$ is even or odd.  See Lemma~\ref{lemma}.
We use this simplification when computing the composite rule for
uniform partitions.  The polynomial $\phint$ can be replaced with
any of the minimizing polynomials from Corollaries~\ref{corollaryone},
\ref{corollarytwo} or \ref{corollaryinfinity}.
\begin{corollary}\label{corollarycomposite}
Let $\phint\in\polyn$ such that $\phint$ is even or odd as $n$ is even or odd.
The composite rule for a uniform partition, $x_i=a+(b-a)i/m$, $0\leq i\leq m$, is
\begin{align}
&\intab f(x)\,dx  = \notag\\
&\frac{1}{n!}\sum_{\ell=0}^{\lfloor(n-1)/2\rfloor}\left(\frac{b-a}{2m}\right)^{2\ell+1}
\phint^{(n-2\ell-1)}(1)\left[f^{(2\ell)}(a)+f^{(2\ell)}(b)+2\sum_{i=1}^{m-1}f^{(2\ell)}(x_i)
\right]\label{sum1}\\
&+\frac{1}{n!}\sum_{\ell=1}^{\lfloor n/2\rfloor}\left(\frac{b-a}{2m}\right)^{2\ell}
\phint^{(n-2\ell)}(1)\left[f^{(2\ell-1)}(a)-f^{(2\ell-1)}(b)
\right] + E_n^m(f),\label{sum2}
\end{align}
where 
\begin{equation}
\abs{E_n^m(f)}
\leq K_{n,p}\norm{\fdn}_p(b-a)^{n+1/q}m^{-n}.\label{compositeerrorestimate}
\end{equation}
\end{corollary}
\begin{proof}
	Write $\intab f(x)\,dx \doteq \sum_{i=1}^m\int_{x_{i-1}}^{x_i}f(x)\,dx$ and
use the approximation in the Theorem on each interval $[x_{i-1},x_i]$.  
Scaling to 
$[x_{i-1},x_i]$ gives
$$
\phi_{n,i}(x)=\left(\frac{b-a}{2m}\right)^n\phint\left(\frac{2mx-2ma-2(b-a)i+b-a}{b-a}\right).
$$
Upon changing summation order, this gives
\begin{align*}
&\intab f(x)\,dx \doteq\\
&\frac{1}{n!}\sum_{k=0}^{n-1}(-1)^{k+1}\left(\frac{b-a}{2m}\right)^{k+1}
\sum_{i=1}^m\left[\phint^{(n-k-1)}(-1)f^{(k)}(x_{i-1})-
\phint^{(n-k-1)}(1)f^{(k)}(x_i)\right].
\end{align*}
We have that $\phint^{(n-k-1)}$ is even whenever
$k$ is odd, and is odd whenever $k$ is even.  Hence, the sum on $i$ telescopes
when $k$ is odd.

The error is written
$$
\abs{E_n^m(f)}=\frac{1}{n!}\left|\intab f^{(n)}(x)\psi_n(x)\,dx\right|
\leq \frac{\norm{f^{(n)}}_p\norm{\psi_n}_q}{n!},
$$
where $\psi_n(x)=\phi_{n,i}(x)\chi_{(x_{i-1},x_i)}(x)$.  And,
with $\Delta x=(b-a)/(2m)$,
\begin{eqnarray*}
\norm{\psi_n}_q & = & \left(\sum_{i=1}^m\int_{x_{i-1}}^{x_i}\abs{\phi_{n,i}
(x)}^q\,dx\right)^{1/q}\\
 & = & \left(\frac{b-a}{2m}\right)^n
\left(\sum_{i=1}^m\int_{x_{i-1}}^{x_i}\left|\phint\left(\frac{
2x-x_{i-1}-x_i}{\Delta x}\right)\right|^q\,dx\right)^{1/q}\\
 & = & \left(\frac{b-a}{2m}\right)^{n+1/q}\left(\sum_{i=1}^m\int_{-1}^1\abs{
\phint(x)}^q\,dx\right)^{1/q}
=  \left(\frac{b-a}{2}\right)^{n+1/q}\norm{\phint}_q m^{-n}.
\end{eqnarray*}

The constant is proved sharp as in the Theorem.  In the case $p=1$ the
same method works since at each point where $\psi_n$ attains its
extrema, it is either continuous from the left or from the right and we
can take a $\delta$-sequence supported on the left or right.
\end{proof}

When $n=1$, the sum in \eqref{sum2} is absent and we get the usual composite trapezoidal rule
$$
\intab f(x)\,dx\doteq \frac{b-a}{2m}\left[f(a)+f(b)+2\sum_{i=1}^{m-1}f(x_i)\right].
$$
When $n=2$, the sum in \eqref{sum1} contains only the $\ell=0$ term.  Derivatives of $f$ then 
appear as $f'(a)$ and $f'(b)$ but not at interior nodes, cf. \cite{talvilawiersmajmi}.

If $\fdn\in L^p([a,b])$ for a value of $p$ different from $1$, $2$,
or $\infty$ we can use the norm $\norm{\fdn}_p$ in \eqref{errorestimate}
and estimate the value of $\norm{\phi_n}_q$, even though the minimizing
polynomial is not known.  This gives an estimate of $\Knpt$. Together
with
the exact values of $\Knpt$ from Corollaries~\ref{corollaryone},
\ref{corollarytwo}, \ref{corollaryinfinity}, these give the
asymptotic behaviour
of $\Knpt$ as $n\to\infty$, uniformly
valid for all $p\in[1,\infty]$.

\begin{prop}\label{propestimatephin}
Let $\Knpt$ be the constant from 
Theorem~\ref{theoremapprox}.  Then ${\tilde K_{n,1}}=2^{1-2n}/n!$,
${\tilde K_{n,2}}=n!(2n+1)^{-1/2}/(2n)!$, ${\tilde K_{n,\infty}}=2^{-2n}/n!$.
For each fixed $n$, $K_{n,p}$ is a decreasing function of $p$.  Hence,
$
K_{n,p}\leq K_{n,1}$.
\end{prop}
\begin{proof}
The values of $\Knpt$ for $p=1, 2, \infty$ are given in 
Corollaries~\ref{corollaryone},
\ref{corollarytwo}, \ref{corollaryinfinity}, respectively.

If $q<r<\infty$ then Jensen's inequality gives
$$
\norm{\phint}_q^r=\left(2\int_{-1}^1\abs{\phint(x)}^q\,\frac{dx}{2}
\right)^{r/q}\leq 2^{r/q}\int_{-1}^1\abs{\phint(x)}^r\,\frac{dx}{2}
=2^{r/q-1}\norm{\phint}_r^r.
$$
And, $\norm{\phint}_q\leq 2^{1/q-1/r}\norm{\phint}_r$.

Let $1\leq p_1<p_2\leq\infty$ with corresponding conjugate exponents
$q_1,q_2$.  Let $\phint$ be the minimizing polynomial for $\norm{\cdot}_{q_1}$.
Then $$
K_{n,p_2}\leq\frac{\norm{\phint}_{q_2}}{n!2^{n+1/q_2}}\leq\frac{2^{1/q_2-1/q_1}
\norm{\phint}_{q_1}}{n!2^{n+1/q_2}}=K_{n,p_1}.
$$
Hence, $K_{n,p}$ is decreasing.
\end{proof}

\section{Lemma on minimizing polynomials}
For each of the $L^p$ norms there is a unique monic polynomial that
minimizes the norm.
Define $F_q\fn\polym\to\R$ by $F_q(\phi)=\norm{\phi}_q$ where 
$1\leq q\leq \infty$
and the norms are over compact interval $[a,b]$.
Since $F_q(\phi)$ is bounded below for $\phi\in\polym$ it has
an infimum over $\polym$.  It also has a unique minimum at a polynomial that
has $m$ roots in $[a,b]$.  As well, the error-minimizing polynomial is
even or odd about the midpoint of $[a,b]$ as $m$ is even or odd.

\begin{lemma}\label{lemma}
(a) For $m\geq 2$, let $\phi\in\polym$ with a non-real root.  There  exists
$\psi\in\polym$ with a real root such that $F_q(\psi)<F_q(\phi)$.
(b) Let $\phi\in\polym$ with a root $t\not\in[a,b]$.  
There  exists $\psi\in\polym$ with a root in $[a,b]$ such that 
$F_q(\psi)<F_q(\phi)$.
(c) If $\phi$ minimizes $F_q$ then it has $m$ simple zeros
in $[a,b]$.
(d) If $F_q$ has a minimum in $\polym$ it is unique.
(e) $F_q$ attains its minimum over $\polym$.
(f) If $\phi\in\polym$ is neither even nor odd about $c:=(a+b)/2$ then
there is a polynomial $\psi\in\polym$ that is either even or odd
about $c$ such that $F_q(\psi)< F_q(\phi)$.
(g) The minimum of $F_q$ occurs
at a polynomial $\phi\in\polym$ with $m$ simple zeros
in $[a,b]$.  If $m$ is even about $c$ then
so is $\phi$.  If $m$ is odd about $c$ then so is $\phi$.
This minimizing polynomial is unique.
(h) Suppose $\phi\in\polym$ is a minimum of $F_\infty$.  Then
$\phi(x)=\prod_{i=1}^m(x-t_i)$ for $a< t_1<t_2
<\cdots<t_m< b$.  For each $1\leq i\leq m-1$
there is
$\xi_i\in(t_i,t_{i+1})$ such that 
$\abs{\phi(\xi_i)}=\norm{\phi}_\infty$.
\end{lemma}
This result is proved in \cite{talvilawiersmajmi}.  See also
\cite{cheney}, \cite{davis}, \cite{devorelorentz}, \cite{golomb},
\cite{lorentz}, \cite{natansonI}, \cite{natansonII}, \cite{rivlin},
\cite{steffens}, \cite{timan} for background on this problem and 
references to original papers by Bernstein, Chebyshev, Jackson,
etc.

The cases $q=1,2,\infty$ are used in Corollaries~\ref{corollaryinfinity},
\ref{corollarytwo}, \ref{corollaryone}, respectively.  Here, the
minimizing polynomials are orthogonal polynomials.  No explicit solutions
appear to be known for any other values of $q$.   Gillis and Lewis
\cite{gillislewis} give a heuristic argument to show that for no other
values of $q$ are the minimizing polnomials a family of orthogonal polynomials.

\section{Alexiewicz norm}
The Alexiewicz norm is useful for functions or distributions for which
$\intab f(x)\,dx$ exists but $\intab\abs{f(x)}\,dx$ diverges.  It is
defined as $\norm{f}=\sup_{a\leq x\leq b}\abs{\int_a^x f(t)\,dt}$.  If
$F\in C([a,b])$ with $F(a)=0$ then define $\alexc$ to be the
Schwartz distributions, $f$, for which $F'=f$.  The derivative is
understood in the distributional sense, $\langle F',\phi\rangle
=-\langle F,\phi'\rangle =-\int_a^b F(x)\phi'(x)\,dx$ where $\phi\in
C^\infty_c((a,b))$ (smooth functions with compact support in $(a,b)$).  Then
$\alexc$ is a Banach space isometrically isomorphic to the continuous
functions on $[a,b]$ that vanish at $a$, and $\norm{f}=\norm{F}_\infty$
where $F$ is the unique primitive of $f$.  This integration process
is often called the {\it continuous primitive integral} and
$\int_a^x f(x)\,dx=F(x)$ for all $x\in[a,b]$.  See \cite{talviladenjoy}
for details.  Since this integral uses the space of all continuous
functions as primitives, it includes the Lebesgue integral (whose
primitives are absolutely continuous) and the Henstock--Kurzweil
integral (whose primitives are continuous but need not be
absolutely continuous and are described in \cite{lee}).

An example of a function integrable in the Henstock--Kurzweil sense but
not in $L^1([-1,1])$ is given by $f(x)=F'(x)$ where
$$
F(x)=\left\{\begin{array}{cl}
x^2\sin(x^{-3}), & x\not =0\\
0, & x=0.
\end{array}
\right.
$$
The pointwise derivative $F'(x)$ exists at each point.
If we take $F$ to be a continuous monotonic function whose derivative is
zero almost everywhere then the Lebesgue integral $\intab F'(x)\,dx=0$ while
the continuous primitive integral gives $\intab F'=F(b)-F(a)$.  If $f$ is
a continuous function differentiable nowhere in
the pointwise sense then the distributional derivative $F'\in\alexc$
and $\int_a^xF'=F(x)-F(a)$ for all $x\in[a,b]$, even though $F'$ is a
distribution that does not have any pointwise values and the Lebesgue
integral of $F'$ is meaningless.  Additional
examples are given in \cite{talviladenjoy}.

The integration by parts formula for $f\in\alexc$ with primitive $F$
and function $g$ of bounded variation is given in terms of a Riemann--Stieltjes
integral
\begin{equation}
\intab f(x)g(x)\,dx = F(b)g(b)-F(a)g(a)-\intab F(x)dg(x)\label{alexparts}.
\end{equation}
The H\"older inequality is then
\begin{equation}
\left|\intab f(x)g(x)\,dx\right|\leq
\norm{f}(\abs{g(b)}+ V\!g).\label{alexholder}
\end{equation}
See \cite[Theorem~12.3]{lee}.

\begin{theorem}\label{theoremalex}
Let $n\in\N$.
Let $f\fn[a,b]\to\R$ such that the pointwise derivative $f^{(n-1)}$ is continuous.
Then the distributional derivative $f^{(n)}\in\alexc$.
Write
\begin{align}
&\intab f(x)\,dx  = \notag\\
&\frac{1}{(n-1)!}\sum_{k=0}^{n-2}\left[
\frac{(b-a)^{k+1}(2n-k-2)!(n-k-1)!}{2^{2k+1}(2n-2k-2)!(k+1)!}\right]
\left[
f^{(k)}(a)+(-1)^{k}f^{(k)}(b)\right]\notag\\
&    \quad +\left\{\begin{array}{cl}
\frac{(b-a)^nf^{(n-1)}(a)}{n!2^{2n-2}}, &  n \text{ odd}\label{alexbasic}\\
0, & n \text{ even}
\end{array}
\right.
+E_n(f).
\end{align}
Then 
\begin{equation}
\abs{E_n(f)}\leq
\frac{\norm{\fdn}(b-a)^{n}}{(n-1)!2^{2n-2}}.\label{alexerrorestimate}
\end{equation}
The estimate on $\abs{E_n(f)}$ is then sharp in the sense that the
coefficient of $\norm{\fdn}$ cannot be reduced.
\end{theorem}
\begin{proof}
Let $\phin\in\polyn$.
From \eqref{basic} and the H\"older inequality \eqref{alexholder} we have
$\abs{E_n(f)}\leq\norm{\fdn}(\abs{\phin(b)}+V\!\phin)$.  Note that
$V\!\phin=\intab\abs{\phin'(x)}\,dx=\norm{\phin'}_1=n\norm{\psi_{n-1}}_1,$
where $\psi_{n-1}=\phin'/n$.  Hence, to minimize the variation over 
$\phin\in\polyn$
we minimize the one-norm over $\psi_{n-1}\in\poly{{n-1}}$.  Adding a
constant to $\phin$ does not affect the variation so the unique minimizing
polynomial on $[-1,1]$ is (cf. proof of Corollary~\ref{corollaryinfinity})
\begin{eqnarray}
\phint(x) & = & -n2^{1-n}\int_x^1U_{n-1}(t)\,dt\label{alexphint}\notag\\
 & = & -2^{1-n}\int_x^1T_n'(t)\,dt=
2^{1-n}[T_n(x)-1].\label{alexminpoly}
\end{eqnarray}
The minimizing polynomial on $[a,b]$ is 
$$
\phin(x)=2^{1-2n}\left(b-a\right)^{n}\left[T_n\left(\frac{2x-a-b}{b-a}\right)-1
\right].
$$
We have $\phin(b)=0$ and 
\begin{eqnarray*}
V\!\phin & = & 2^{2-2n}(b-a)^{n-1}\intab\left|T_n'\left(\frac{2x-a-b}{b-a}\right)
\right|\,dx\\
 & = & 2^{1-2n}(b-a)^{n}\norm{T_n'}_1
=  2^{1-2n}(b-a)^{n}n\norm{U_{n-1}}_1=2^{2-2n}(b-a)^nn.
\end{eqnarray*}
The H\"older inequality \eqref{alexholder} now gives the estimate in 
\eqref{alexerrorestimate}.

Formula \eqref{alexbasic} is computed as in the proof of 
Corollary~\ref{corollaryone}.

Now show there can be equality in \eqref{alexerrorestimate}.  Using
\eqref{alexphint}, integrate by parts \eqref{alexparts} to get
\begin{eqnarray*}
E_n(f) & = & \frac{(-1)^n}{n!}\int_{-1}^1 \fdn(x)\phint(x)\,dx\\
 & = & \left\{\begin{array}{cl}
-\frac{2^{1-n}}{(n-1)!}\int_{-1}^1 f^{(n-1)}(x)U_{n-1}(x)\,dx, & n \text{ even}\\
-\frac{2^{2-n}f^{(n-1)}(-1)}{n!}+\frac{2^{1-n}}{(n-1)!}\int_{-1}^1 f^{(n-1)}(x)U_{n-1}(x)\,dx, & n \text{ odd.}
\end{array}
\right.
\end{eqnarray*}
By the usual H\"older inequality,
\begin{eqnarray*}
\frac{2^{1-n}}{(n-1)!}\left|\int_{-1}^1 f^{(n-1)}(x)U_{n-1}(x)\,dx\right|
 & \leq & \frac{2^{1-n}}{(n-1)!}\norm{f^{(n-1)}}_\infty\norm{U_{n-1}}_1\\
 & = & \frac{2^{2-n}}{(n-1)!}\norm{f^{(n-1)}}_\infty.
\end{eqnarray*}
As in the proof of Theorem~\ref{theoremapprox} there is equality when
$f^{(n-1)}(x)=d\,{\rm sgn}[U_{n-1}(x)]$ for some $d\in\R$.  Since the final
term in \eqref{alexparts} can then have any sign, we integrate to get
$$
f(x)=\frac{d}{(n-2)!}\int_{-1}^x(x-t)^{n-2}{\rm sgn}[U_{n-1}(t)]\,dt
$$
for $n\geq 2$, modulo a polynomial of degree at most $n-2$ that vanishes
at $-1$. 
For $n=1$, $f(x)=d(x+1)$.  Hence, the coefficient
in \eqref{alexerrorestimate} cannot be reduced.
\end{proof}

When $n=1$ the minimizing polynomial is $\phi_1(x)=x-b$.  
The approximation becomes $\intab f(x)\,dx\doteq f(a)(b-a)$ with
error $\abs{E_1(f)}\leq \norm{f'}(b-a)$.

When $n=2$ the minimizing polynomial is $\phi_2(x)=(x-a)(x-b)$.
The approximation reduces to the usual trapezoidal rule, 
$\intab f(x)\,dx\doteq (b-a)[f(a)+f(b)]/2$ with
error $\abs{E_2(f)}\leq \norm{f''}(b-a)^2/4$.  This appears as Theorem~5.1
in \cite{talvilawiersmajmi}.  An alternate form of the Alexiewicz norm
is also considered in this paper.

If $f^{(n)}\in\alexc$ then $f^{(n-1)}\in L^\infty([a,b])$ and the results
of Theorem~\ref{theoremalex} agree with Corollary~\ref{corollaryinfinity}
with $n$ reduced by one.

\begin{corollary}\label{corollarycompositealex}
The composite rule for a uniform partition, $x_i=a+(b-a)i/m$, $0\leq i\leq m$, is
\begin{align*}
&\intab f(x)\,dx  \doteq
\frac{1}{(n-1)!}\sum_{\ell=0}^{\lfloor(n-2)/2\rfloor}\left(\frac{b-a}{m}\right)^{2\ell+1}
\frac{(2n-2\ell-2)!(n-2\ell-1)!}{2^{4\ell+1}(2n-4\ell-2)!(2\ell+1)!}\ast\\
&\qquad\qquad\qquad\qquad\qquad\qquad\qquad\qquad\qquad\ast\left[f^{(2\ell)}(a)+f^{(2\ell)}(b)+2\sum_{i=1}^{m-1}f^{(2\ell)}(x_i)
\right]\label{sum1alex}\\
&+\frac{1}{(n-1)!}\sum_{\ell=1}^{\lfloor(n-1)/2\rfloor}\left(\frac{b-a}{m}\right)^{2\ell}
\frac{(2n-2\ell-1)!(n-2\ell)!}{2^{4\ell-1}(2n-4\ell)!(2\ell)!}
\left[f^{(2\ell-1)}(a)-f^{(2\ell-1)}(b)
\right]\\
&+\left\{\begin{array}{cl}
\frac{(b-a)^n}{n!2^{2n-2}m^n}
\sum_{i=0}^{m-1}f^{(n-1)}(x_i), & n \text{ odd}\\
0, & n \text{ even.}
\end{array}
\right.
\end{align*}
A sharp estimate for the error is
\begin{equation}
\abs{E_n^m(f)}
\leq \frac{\norm{\fdn}(b-a)^{n}}{m^{n-1}(n-1)!2^{2n-2}}.\label{compositeerrorestimatalex}
\end{equation}
\end{corollary}
\begin{proof}
Using the minimizing polynomial \eqref{alexminpoly}, the proof of the
approximation formula is similar to the proof of 
Corollary~\ref{corollarycomposite}.  With the notation of that corollary,
\begin{eqnarray*}
V\!{\psi_n} & = & \sum_{i=1}^m\int_{x_{i-1}}^{x_i}\abs{\phi_{n,i}'
(x)}\,dx
  =  \left(\frac{b-a}{2m}\right)^{n-1}
\sum_{i=1}^m\int_{x_{i-1}}^{x_i}\left|\phint'\left(\frac{
2x-x_{i-1}-x_i}{\Delta x}\right)\right|\,dx\\
 & = & \left(\frac{b-a}{2m}\right)^{n}\sum_{i=1}^m\int_{-1}^1\abs{
\phint'(x)}\,dx
=  \left(\frac{b-a}{2m}\right)^{n}mV\!(2^{1-n}T_n)\\
 & = & \left(\frac{b-a}{2m}\right)^{n}m2^{1-n}n\norm{U_{n-1}}_1
=  \frac{(b-a)^nn}{m^{n-1}2^{2n-2}}.
\end{eqnarray*}
Formula \eqref{compositeerrorestimatalex} now follows from the
H\"older inequality \eqref{alexholder}. 

\end{proof}

\section{Degree of exactness}
We now show that formula \eqref{basic} is exact for all $\phin\in\polyn$
when $f$ is a polynomial of degree at most $n-1$ and this formula is
exact for $f$ being a polynomial of degree at most $2n-1$ when
$\phin$ is the composition of a linear function and a Legendre polynomial.
\begin{theorem}
Let $f\in C^{n-1}([a,b])$.  Write
\begin{eqnarray}
\intab f(x)\,dx & = & \frac{(-1)^n}{n!}\sum_{k=0}^{n-1}(-1)^{n-k-1}\left[
\phi_n^{(n-k-1)}(a)f^{(k)}(a)- \phi_n^{(n-k-1)}(b)f^{(k)}(b)\right]\notag\\
 & & \qquad+E_n(f),\label{basicexact}
\end{eqnarray}
where
$E_n(f)=\frac{(-1)^n}{n!}\intab f^{(n)}(x)\phi_n(x)\,dx$ and $\phin\in\polyn$.
(a)  If $f$ is a polynomial of degree at most $n-1$ then \eqref{basicexact} 
is exact ($E_n(f)=0$) for all $\phin\in\polyn$.
(b)  If $f$ is a polynomial of degree at most $2n-1$ then \eqref{basicexact} 
is exact ($E_n(f)=0$) 
if and only if $$
\phin(x)=\frac{(b-a)^n(n!)^2}{(2n!)}P_n\left(\frac{2x-a-b}{b-a}\right),
$$
where $P_n$ is a Legendre polynomial.
\end{theorem}
\begin{proof}
(a) Integrate by parts.  (b) First consider the interval $[-1,1]$.  By
(a) and linearity we need only consider $f(x)=\sum_{i=n}^{2n-1}A_ix^i$
for real numbers $A_i$.  Let $\phi\in\polyn$.  We then require
$$
(-1)^nn!E_n(f)  =  \sum_{i=n}^{2n-1}A_ii(i-1)\cdots(i-n+1)\int_{-1}^1 x^{i-n}
\phin(x)\,dx=0.
$$
But the terms $A_ii(i-1)\cdots(i-n+1)$ are linearly independent so we
require the moments $\int_{-1}^1 x^{j}
\phin(x)\,dx$ to vanish for $0\leq j\leq n-1$.  The Legendre polynomial
$P_n$ satisfies this integral condition \cite[7.222.1]{gradshteyn}.
Legendre polynomials are orthogonal with respect to the inner product
$\langle g,h\rangle=\int_{-1}^1f(x)g(x)\,dx$.
If there was another polynomial $\psi_n\in\polyn$ that also satisfied
the moment condition then it could be expanded as $\psi_n(x)=
\sum_{k=0}^nB_kP_k(x)$ where $B_k=\langle \psi_n,P_k\rangle/\langle P_k,
P_k\rangle$.  The moment condition gives
$$
0=\int_{-1}^1x^j\psi_n(x)\,dx=\sum_{k=0}^jB_k\int_{-1}^1x^jP_k(x)\,dx
$$
for each $0\leq j\leq n-1$.   Putting $j=0$ gives 
$0=B_0\int_{-1}^1P_0(x)\,dx$.  But 
$\int_{-1}^1x^mP_m(x)\,dx=(m!)^22^{m+1}/(2m+1)!\not=0$ 
\cite[7.224.3]{gradshteyn}.  So, $B_0\not=0$.  Successively
putting $j=1,2,3,\ldots$ now shows
$B_k=0$ for each $0\leq k\leq n-1$.  Hence,
$\psi_n$ is a multiple of $P_n$.
A linear transformation gives the required polynomial on $[a,b]$.
\end{proof}

\end{document}